\documentclass[journal]{IEEEtran}

\usepackage{graphicx}
\usepackage[cmex10]{amsmath}
\usepackage{amsthm}
\usepackage{graphicx}%,color,ulem} % for pdf, bitmapped graphics files
\usepackage{amssymb}  % assumes amsmath package installed
\usepackage{array}
\usepackage{multirow}
\usepackage{color}
\usepackage[backref=none]{hyperref}

\newtheorem{thm}{Theorem}
\newtheorem{proposition}[thm]{Proposition}
\newtheorem{lemma}[thm]{Lemma}
\newtheorem{remark}{Remark}
\newtheorem{definition}{Definition}
\newtheorem{assumption}{Assumption}
\usepackage{mdwmath}
\usepackage{mdwtab}
\usepackage{color}

\usepackage{eqparbox}
\usepackage{url}
\begin{document}
\title{Approximate controllability of the Schr\"{o}dinger Equation with a polarizability term in higher Sobolev norms
\thanks{This  work has been supported by the INRIA Nancy-Grand Est ``CUPIDSE'' Color
program.
The work of M. Caponigro and T. Chambrion was partially supported by French Agence National de
la Recherche ANR ``GCM'', program ``BLANC-CSD'', contract number NT09-504590.
The work of T. Chambrion was partially supported by European Research Council ERC StG
2009 ``GeCoMethods'', contract number 239748. 
}
}

\author{Nabile~Boussa\"{i}d,
        Marco~Caponigro,
        and~Thomas~Chambrion% <-this % stops a space
\thanks{N. Boussa\"{i}d is with  Universit\'e de Franche--Comt\'e,
Laboratoire de math\'ematiques, 16 route de Gray, 25030
Besan\c{c}on Cedex, France.}% <-this % stops a space
\thanks{M. Caponigro is with  Conservatoire National des Arts et M\'etiers, \'Equipe M2N, 292 rue Saint-Martin, 75003, Paris, France.}%
\thanks{T. Chambrion is with  Universit\'e de Lorraine, Institut \'Elie Cartan de Lorraine, UMR 7502, 
Vand{\oe}uvre-l\`es-Nancy, F-54506, France, CNRS, Institut \'Elie Cartan de Lorraine, UMR 7502, Vand{\oe}uvre-l\`es-
Nancy, F-54506, France and Inria, Villers-l\`es-Nancy, F-54600, France.}% <-this % stops a space
}
\maketitle

%%%%%%%%%%%%%%%%%%%%%%%%%%%%%%%%%%%%%%%%%%%%%%%%%%%%%%%%%%%%%%%%%%%%%%%%%%%%%%%%
\begin{abstract}
This analysis is concerned with the controllability  of quantum systems in 
the case where the standard dipolar approximation, involving the permanent dipole moment of the system, 
 is corrected with a polarizability term, involving the field induced dipole moment. 
Sufficient conditions for approximate controllability are given. For transfers between eigenstates of the free Hamiltonian, the control laws are explicitly given. The results apply also for unbounded or non-regular potentials.
\end{abstract}

%%%%%%%%%%%%%%%%%%%%%%%%%%%%%%%%%%%%%%%%%%%%%%%%%%%%%%%%%%%%%%%%%%%%%%%%%%%%%%%%
\section{INTRODUCTION}

\subsection{Control of quantum systems}

The state of a quantum system evolving on a Riemannian manifold $\Omega$ is described by its wavefunction 
$\psi$, an element of {the unit sphere of } $L^2(\Omega,\mathbf{C})$.
When the system is submitted to an electric field, 
the time evolution of the wavefunction is given by the Schr\"{o}dinger equation
\begin{equation}\label{EQ_bilinear_Schrod}
\mathrm{i}\frac{\partial \psi}{\partial t}=(-\Delta +V(x))\psi + \mu(u,x) \psi(t), \quad x\in \Omega,
\end{equation}
where $\Delta$ is the Laplace--Beltrami operator on $\Omega$, $V:\Omega \to \mathbf{R}$ is a potential describing the evolution of the
system in absence of control, $u$ is the scalar  function depending on time and modeling the intensity of the electric field and 
$\mu:  \mathbf{R} \times  \Omega\to \mathbf{R}$ describes the effect of the external field.
In the dipolar approximation
we expand $\mu$ to the first order in $u$ and we then represent 
$\mu(u,x)$ as $uW(x)$,
where $W$ is a real valued function.

Although the dipolar approximation usually gives excellent results for low intensity fields, it is sometimes necessary, 
when dealing with stronger fields, to consider a better approximation of $\mu$ involving the first two terms of its expansion 
in $u$. Therefore an approximation of
$\mu(u,x)$ by $u W_1(x) + u^2 W_2(x)$, for two real functions $W_{1}(x)$ and $W_{2}(x)$, gives a more accurate representation of the external field.  The need for a modeling involving the quadratic term appears, for instance, in the control of orientation 
of a rotating HCN molecule,  \cite{DBAKU} and \cite{DBAKU2}.

The aim of this work is to present controllability properties for the controlled Schr\"{o}dinger equation, using the dipolar term $u W_1$ and the polarizability term $u^{2}W_{2}$. 

This question has already been tackled by various authors in \cite{Coron,Grigoriou} (for finite dimensional approximations) 
and in \cite{Morancey} (for the infinite dimensional version of the problem, when $\Omega$ is a bounded set of $\mathbf{R}^n$ and $W_1,W_2$ are smooth functions).  All the results in these contributions rely on Lyapunov methods. 

The novelty of our contribution is the use of geometric methods inspired by finite 
dimensional geometric control theory \cite{book}, in the spirit of~\cite{Schrod} and \cite{Schrod2}. 
This point of view allows us  to  state the first available positive approximate 
controllability results for system (\ref{EQ_bilinear_Schrod}) in the case where the  
potentials $W_1$ and $W_2$ are unbounded or noncontinuous. Moreover, when considering the 
physically relevant problem of 
transferring the quantum system from an energy level to another,  our method is 
constructive and provides simple fully explicit control laws.

A shorter and simplified version of this analysis has been presented in 
51$^{\rm st}$ Conference on Decision 
and Control (see~\cite{6426619}). In this work, we present several 
extensions with respect to the
 proceeding. The main results have been sensibly improved, providing approximate controllability 
 in higher regularity norms, improved upper bound of the $L^1$ norm of the controls 
 and  approximate controllability between eigenstates 
 coupled by a non-trivial chain of connectedness. Moreover, two applications to rather general 
 examples are discussed.

\subsection{Framework and  notations}\label{SEC_framework}

In order to exploit the powerful tools of functional analysis,  
we set the problem in a  more abstract framework. In a separable Hilbert space $H$, endowed with the Hermitian product 
$\langle \cdot , \cdot \rangle$, we consider the following control system 
\begin{equation}\label{EQ_main}
\frac{d}{dt} \psi=(A+u(t) B +u^2(t) C) \psi,
\end{equation}
where $(A,B,C,k)$ satisfies Assumption \ref{ASS_1} for some $k$.

\begin{assumption}\label{ASS_1}
$k$ is a positive number and
$(A,B,C)$ is a triple of (possibly unbounded) linear operators in $H$ such that
\begin{enumerate}
\item $A$ with domain $D(A)$ is skew-adjoint, with pure point spectrum 
$(-\mathrm{i}\lambda_j)_{j \in \mathbf{N}}$ 
with  $\lambda_{j+1}> \lambda_j>0$ for every $j$ in $\mathbf{N}$ and $\lim_{j\to \infty}\lambda_j 
= \infty$ ;
\item for every $(u_1,u_2)$ in $\mathbf{R}^2$, $A+u_1 B +u_2 C$ is 
skew-adjoint with domain $D(A)$;
\item 
for every $(u_1,u_2)$ in $\mathbf{R}^2$, $|A+u_1 B +u_2 C|^{k/2}$ has domain 
$D(|A|^{k/2})$;
\item $\displaystyle{\!\!\!\!\!\sup_{\psi \in D(|A|^k)\setminus\{0\}}\!\!\!
\left(\frac{ |\Re \langle |A|^k
\psi,B \psi \rangle |}{ |\langle |A|^k \psi, \psi \rangle|} \!+\!\frac{ |\Re \langle |A|^k
\psi,C \psi \rangle |}{ |\langle |A|^k \psi, \psi \rangle|}\!
\right)\!\!\!<\!+\infty};$
%there exists
%a constant $c_{(A,B,C)}$ such that, for every $\psi$ in $D(|A|^k)$, $ |\Re \langle |A|^k
%\psi,B \psi \rangle |\leq c_{(A,B,C)} |\langle |A|^k \psi, \psi \rangle|$ and
%$ |\Re \langle |A|^k
%\psi, C\psi \rangle |\leq c_{(A,B,C)} |\langle |A|^k \psi, \psi \rangle|$;
\item there
exist $d>0$ and $0\leq r<k$ such that $\|B\psi \|\leq d \||A|^{r/2}\psi \|$ 
and $\|C\psi \|\leq d \||A|^{r/2}\psi \|$ for
every $\psi$ in $D(|A|^{r/2})$. \label{ASS_BC_Ak_borne}
\end{enumerate} 
\end{assumption}
If $(A,B,C,k)$ satisfies Assumption~\ref{ASS_1}, we define the \emph{coupling constant}
$c_{(A,B,C,k)}$ as the lower bound of the set of every real $c$ such that
for every $\psi$ in $D(|A|^k)$, $ |\Re \langle |A|^k
\psi,B \psi \rangle |\leq c |\langle |A|^k \psi, \psi \rangle|$ and
$ |\Re \langle |A|^k \psi, C\psi \rangle |\leq c |\langle |A|^k \psi, \psi \rangle|$.

 From Assumption~\ref{ASS_1} there exists a Hilbert
basis $(\phi_k)_{k\in \mathbf{N}}$ of $H$ made of eigenvectors of $A$. For every
$j$, $A \phi_j =-\mathrm{i} \lambda_j \phi_j$.
Since $A$ is skew-adjoint and diagonalizable in a Hilbert basis $(\phi_k)_{k\in \mathbf{N}}$, $|A|$ is self-adjoint positive and
diagonalizable in the same basis $(\phi_k)_{k\in \mathbf{N}}$. 
The eigenvalues of $|A|$ are the moduli of the eigenvalues of $A$. 
We define the $k$-norm 
of an element $\psi$ of $D(|A|^k)$ as $\|\psi\|_k:=\| |A|^k \psi \|$. 
When $\Omega$ is a compact Riemannian  manifold and $A=\mathrm{i}\Delta$, the $k$-norm is 
equivalent to the Sobolev $H^{2k}(\Omega,\mathbf{C})$ norm on $\Omega$. 
 
In the following, we say that $u:\mathbf{R}\to \mathbf{R}$ is \emph{piecewise constant} if there exists a 
non decreasing sequence $(t_j)_{j\in \mathbf{N}}$ of $\mathbf{R}$ that tends to $+\infty$ such that $u$ is 
constant on $[t_j,t_{j+1})$ for every $j$ in 
$\mathbf{N}$.

If $(A,B,C,k)$ satisfies Assumption \ref{ASS_1}, for every $u$ in $\mathbf{R}$, $A+uB+u^2C$ 
generates a group of unitary propagators $t\mapsto e^{t(A+uB+u^2C)}$. By concatenation, 
one can define the solution of (\ref{EQ_main}) for every piecewise constant $u$, for every 
initial condition $\psi_0$ given at time $t_0$.  We denote this solution $t\mapsto 
\Upsilon^{u,(A,B,C)}_{t,t_0} \psi_0$ or simply $t\mapsto 
\Upsilon^{u}_{t,t_0} \psi_0$ when it does not create ambiguities. 

We will see in Section~\ref{SEC_HIGH_NORM_WC} below that 
the mapping $u\mapsto \Upsilon^u_{T,t_0}\psi_0$ admits a unique continuous extension (for the 
$\| \cdot \|_{L^1}+\| \cdot \|_{L^2}$ norm) to $L^1(\mathbf{R},\mathbf{R})\cap
 L^2(\mathbf{R},\mathbf{R})$, for every fixed $T\geq 0$.

The operators $B$ and $C$ can be seen as infinite dimensional matrices in the basis 
$(\phi_j)_{j \in \mathbf{N}}$. For every $j,l \in \mathbf{N}$,  we denote 
$b_{jl}=\langle \phi_j, B \phi_l \rangle$ and $c_{jl}=\langle \phi_j, C \phi_l \rangle$.
For every $N$, the orthogonal projection $\pi_N:H\rightarrow H$ on the space spanned by 
the first  $N$ eigenvectors of $A$ is defined by
$$
\pi_N(x)=\sum_{l=1}^N \langle \phi_l,x\rangle \phi_l \quad \quad \mbox{for every } x \mbox{ in } H.
$$
Let $\mathcal{L}_N$ be the range of $\pi_N$.  The \emph{compressions} of $A$, $B$ and $C$ 
at order $N$ are the finite rank
operators $A^{(N)}=\pi_N A_{\upharpoonright \mathcal{L}_N}$, $B^{(N)}=\pi_N
B_{\upharpoonright \mathcal{L}_N} $ and $C^{(N)}=\pi_N
C_{\upharpoonright \mathcal{L}_N}$ respectively. 
The \emph{Galerkin approximation}  of (\ref{EQ_main})
of order $N$ is the system
\begin{equation}\label{eq:sigma}
\dot x = (A^{(N)} + u B^{(N)} +u^2 C^{(N)}) x, \quad x \in \mathcal{L}_{N}
\end{equation}

Physically, the gap $\lambda_j-\lambda_k$ represents the amount of energy necessary to jump from 
the energy level $k$ (i.e., the eigenstate $\phi_k$ of $A$ associated with eigenvalue 
$-\mathrm{i}\lambda_k$) to energy level $j$. Our controllability results rely on the 
possibility to excite, independently, different energy gaps 
$\lambda_j-\lambda_k$. More precisely we have the following set of definitions.

\begin{definition}
A pair $(j,l)$ in $\mathbf{N}^2$ is a \emph{weakly}  \emph{non-degenerate transition} of 
$(A,B,C)$ if $|b_{jl}|+|c_{jl}|\neq 0$
and, for every $m,n$, 
$|\lambda_j-\lambda_l|=|\lambda_n-\lambda_m|$ implies $\{j,l\}=\{m,n\}$ or $|b_{mn}|+|c_{mn}|=0$ 
or $\{m,n\}\cap\{j,l\}=\emptyset$.
\end{definition}

\begin{definition}
A 
pair $(j,l)$ in $\mathbf{N}^2$ is a  \emph{strongly} \emph{non-degenerate transition} of 
$(A,B,C)$ if $|b_{jl}|+|c_{jl}|\neq 0$
and, for every $m,n$, 
$|\lambda_j-\lambda_l|=|\lambda_n-\lambda_m|$ implies 
 $\{j,l\}=\{m,n\}$.

\end{definition}
\begin{definition}

A pair $(j,l)$ in $\mathbf{N}^2$ is a  \emph{non-resonant transition} of 
$(A,B,C)$ if $|b_{jl}|+|c_{jl}|\neq 0$
and, for every $m,n$, 
$|\lambda_j-\lambda_l|=|\lambda_n-\lambda_m|$ implies 
 $\{j,l\}=\{m,n\}$ or $|b_{mn}|+|c_{mn}|=0$.
\end{definition}
\begin{definition}

A subset $S$ of $\mathbf{N}^2$ is a \emph{chain of connectedness} of $(A,B,C)$ if there exists 
$\alpha$ in $\mathbf{R}$ such that, for every $m,n \in \mathbf{N}$, there exists a finite 
sequence $s_1=(s_1^1,s_1^2),s_2=(s_2^1,s_2^2),\ldots,s_r=(s_r^1,s_r^2) \in S$ such that $s_{1}^1 = m$, $s_{r}^2 = n$, 
$s_{l}^2=s_{l+1}^1$ for every $l=1,\ldots,r-1$ and  
$\langle \phi_{s_{l}^2},(B+\alpha C) \phi_{s_l^1} \rangle \neq 0$ for every $l=1,\ldots,r$. A 
chain of connectedness $S$ of $(A,B,C)$ is \emph{weakly non-degenerate} 
(resp. \emph{strongly non-degenerate}, resp. \emph{non-resonant})  if every $s$ in $S$ is a weakly 
non-degenerate (resp. strongly non-degenerate, resp. non-resonant) transition of 
$(A,B,C)$.

\end{definition}

\begin{remark}
The notion of non-degenerate transition is central in quantum chemistry for several decades,
see for instance \cite[C-XIII]{cohen77} or \cite{PhysRevA.36.4321}, 
and crucial for our geometric techniques. However, we are still in the early ages of control 
of infinite dimensional 
 semi-linear conservative systems and  the terminology is not completely fixed yet. 
 The notion of  ``non-resonant'' transitions appears in \cite{Schrod2}. 
 What we call in this analysis a ``weakly non-degenerate transition'' has been 
called 
 \emph{non-degenerate} in \cite{periodic}.
 Yet another (much stronger) notion of non-resonant transition appears in \cite{Schrod}. 
 Let us cite the promising  ``Lie-Galerkin''
condition recently introduced in 
 \cite{boscain:hal-00789279} as a 
 possible unifying framework for non-degeneracy in quantum control. 
\end{remark}
The main reason for the introduction of the notion of strongly non-degenerate 
transitions is the following stability result.

\begin{lemma}\label{LEM_perturb_chaine_connexite}
Let $(A,B,C,k)$ satisfy Assumption \ref{ASS_1}. If $S$ is a strongly non-degenerate chain 
of connectedness of $(A,B, C)$, then $S$ is a strongly non-degenerate  chain of 
connectedness of $(A,B+\alpha C,0)$ for almost every $\alpha$ in $\mathbf{R}$. 
In particular $S$ is a non-resonant chain of 
connectedness of $(A,B+\alpha C,0)$ for almost every $\alpha$ in $\mathbf{R}$. 
\end{lemma}

\begin{proof} Let $(p,q) \in S\subset \mathbf{N}^2$ and $\alpha$ be a real number. The transition  $(p,q)$ is 
strongly non-degenerate for $(A,B+\alpha C,0)$ if and only if $b_{pq}+\alpha c_{pq} \neq 0$. Hence, for every $\alpha$ in 
$$
\mathfrak{R}^S=\bigcap_{(j,k)\in S} \{\beta \in \mathbf{R}|b_{jk}+\beta c_{jk}\neq 0\},
$$
$S$ is strongly non-degenerate chain of connectedness of $(A,B+\alpha C,0)$. The set $\mathfrak{R}^S$ is a 
countable intersection of complementary to a point subsets of 
$\mathbf{R}$ with full measure, hence $\mathfrak{R}^S$ has full measure in  
$\mathbf{R}$ as the complementary of a countable set.
\end{proof}

\subsection{Main results}
Our main results consist of sufficient conditions for various notions of approximate 
controllability for system (\ref{EQ_main}).

\begin{thm}\label{THE_main}
 Assume that $(A,B,C,k)$ satisfies Assumption \ref{ASS_1} with $k\geq 1$ and that $(A,B,C)$ 
 admits a strongly non-degenerate chain of connectedness. Then, for every 
 $\varepsilon>0 $, for every $N$ in $\mathbf{N}$, for every unitary operator 
 $\hat{\Upsilon}:H\to H$, 
for almost every $\delta>0$, there exist $T_\varepsilon>0$ and a piecewise constant function
$u_{\varepsilon}:[0,T_{\varepsilon}]\rightarrow \{0,\delta\}$  such that
$$
\| 
\Upsilon^{u_{\varepsilon},(A,B,C)}_{T_{\varepsilon},0}\phi_j-\hat{\Upsilon}
\phi_j\|_{r} <\varepsilon,
$$
for every $j\leq  N$ and for every $r<k/2$.
\end{thm}

\begin{thm}\label{THE_main_weak}
 Assume that $(A,B,C,k)$ satisfies Assumption \ref{ASS_1} with $k\geq 1$ and let $S$ be
 a subset of $\mathbf{N}^2$.
 Let $\delta>0$ be such that $S$ is a weakly non degenerate chain of connectedness of $(A,B+\delta C,0)$. 
 Then, for every 
 $\varepsilon>0 $ and for every $p,q$ in $\mathbf{N}$, there exist 
$T_\varepsilon>0$ and a piecewise constant function $u_{\varepsilon}:[0,T_{\varepsilon}]\rightarrow \{0,\delta\}$  such that
$$
\| \Upsilon^{u_{\varepsilon},(A,B,C)}_{T_{\varepsilon},0}\phi_p-\phi_q\|_r<\varepsilon,
$$
for every $r<k/2$. 
\end{thm}

\begin{thm}\label{THE_main_12}
Assume that $(A,B,C,k)$ satisfies Assumption \ref{ASS_1} with $k\geq 1$ and that $(p,q)$ is 
a weakly non-degenerate transition of $(A,B,C)$. 
Let $\delta >0$ be such that $b_{pq}+ \delta c_{pq}\neq 0$.
Then, for every $\varepsilon>0 $ there exist $T_\varepsilon>0$ and a piecewise constant function 
$u_{\varepsilon}:
[0,T_{\varepsilon}]\rightarrow \{0,\delta\}$  such that
$$
\|u_{\varepsilon}\|_{L^1}\leq \frac{\pi}{|b_{pq}+ \delta c_{pq}|} \mbox{ and } 
 \| 
\Upsilon^{u_{\varepsilon},(A,B,C)}_{T_{\varepsilon},0}
\phi_p-\phi_q\|_r<\varepsilon,
 $$
 for every $r<k/2$.
\end{thm}

\subsection{Content of our analysis}
The first part of this work, Section \ref{SEC_Finite_dim},  concerns  
the proof of some preliminary results in finite dimension. In Section~\ref{SEC_infinite}, we 
provide some consequences of Assumption \ref{ASS_1} in terms of energy estimates, 
definitions of solutions and finite dimensional approximations for the system 
(\ref{EQ_main}) (Section \ref{SEC_HIGH_NORM_WC}). Then, 
we use an infinite dimensional tracking result (Section \ref{SEC_infinite_tracking}) to 
prove Theorems \ref{THE_main}, \ref{THE_main_weak}, and \ref{THE_main_12} first in $H$-norm (Sections 
\ref{SEC_Infinite_proof} and \ref{SEC_RWA}),
and then in $r$-norm (Section \ref{SEC_HIGH_NORM}).  
The results of Section \ref{SEC_infinite} are illustrated with two examples. 
The first one deals with system (\ref{EQ_bilinear_Schrod}) involving bounded but irregular (possibly everywhere 
discontinuous) potentials on a 
compact manifold (Section \ref{SEC_bounded_potentials}) and the second one 
with a perturbation of the quantum harmonic oscillator 
involving unbounded potentials (Section \ref{SEC_pert_harm_oscil}).

%%%%%%%%%%%%%%%%%%%%%%%%%%%%%%%%%%%%%%%%%%%%%%%%%%%%%%%%%%%%%%%%%%%%%%%%%%%%%%%%
\section{FINITE DIMENSIONAL PRELIMINARY RESULTS}\label{SEC_Finite_dim}

We  consider the finite dimensional control problem in $\mathcal{L}_N=\mathrm{span}(\phi_1,\ldots,\phi_N)$
\begin{equation}\label{EQ_Galerkin_N}
\dot x=(A^{(N)} +u(t) B^{(N)})x, \quad x \in \mathcal{L}_N.
\end{equation}
Since $B^{(N)}$ is bounded, for every  locally integrable $u$, we can define the solution  (in the sense of 
Carath\'eodory) $t\mapsto X^u_{(N)}(t,t_0)x_0$ of (\ref{EQ_Galerkin_N}) with initial condition $x_0$  in $\mathcal{L}_N$, at time $t_0$.

\subsection{Time reparameterization}

Our results in the following deal with controls in $L^1(\mathbf{R},\mathbf{R}) 
\cap L^2(\mathbf{R},\mathbf{R})$. We will  prove 
these results for piecewise constant control laws, and then extend by density the results to 
general (not necessarily piecewise constant) controls. To this end, we introduce the
 sets $PC$  of  piecewise constant functions $u$ such that there exists 
two sequences 
$0=t_1 <t_2 <\ldots < t_{p+1}$ and $u_1,u_2,\ldots,u_p\neq 0$ with
$$
u=\sum_{j=1}^p u_j \mathbf{1}_{[t_j,t_{j+1})}.
$$
Set $\tau_j=t_{j+1}-t_j$, we identify a function $u$ in $PC$ with the pair $(u_j,\tau_j)_{1\leq j\leq p}$. 

We define similarly $PC^+$ as the set of functions of $PC$ that do not assume negative value:
$$u=\sum_{j=1}^p u_j \mathbf{1}_{[t_j,t_{j+1})} \in PC^+ \Leftrightarrow u_j>0 \quad  \forall j\leq p. $$

We define the mapping $\mathcal{P}:PC^+\rightarrow PC^+$ by
$$
\mathcal{P}\left ((u_j,\tau_j)_{1\leq j \leq p}\right )=\left (\frac{1}{u_j}, u_j \tau_j \right )_{1\leq j \leq p}
$$
for every $u=(u_j,\tau_j)_{1\leq j \leq p}$ in $PC^+$. 

For every $u\in PC$, let 
$P^u$ be the cumulative function of
 $\mathcal{P}|u|$  vanishing at $0$, that is $P^u(t) =
\int_{0}^{t}\mathcal{P}|u|(s)\mathrm{d}s$. By construction, $\int_0^{P^u(t)} 
|u(s)|\mathrm{d}s=t$ for every $t$ in $[0,\|u\|_{L^1}]$.

The mapping $\mathcal{P}$ is a reparameterization of the time with the $L^1$ 
norm of the control.
Indeed, let  $\widehat{X}^u_{(N)}(t,s)$ be the propagator of 
$
\dot x=\mathcal{P}|u|A^{(N)}x + \mathrm{sign}(u\circ P^u) B^{(N)}x,
$
we have the following result.
\begin{lemma}
 For every $u$ in $PC$, 
\begin{equation}\label{eq:0101}
\displaystyle{\widehat{X}^{u}_{(N)}\left (\int_0^T |u(\tau)| \mathrm{d}\tau,0 \right )=X^u_{(N)}(T,0)}.
\end{equation}
\end{lemma}
\begin{proof}
For every constant $\alpha \in \mathbf{R}\setminus \{0\}$,
\begin{eqnarray*}
\lefteqn{\exp(t(A^{(N)}+\alpha B^{(N)}))=}\\
 &\quad \quad \quad &\exp \left (t|\alpha| \left(\frac{1}{|\alpha|} 
A^{(N)}+\mathrm{sign}(\alpha)B^{(N)}\right) \right ).\qedhere
\end{eqnarray*} 
\end{proof}

\subsection{A tracking result}

Lemma \ref{LEM_recurrence} below is an easy consequence of the celebrated Poincar\'e recurrence theorem, see for 
instance \cite{PhysRev.107.337}. Due to the central role it plays in our 
analysis, we present below an elementary proof.
\begin{lemma}\label{LEM_recurrence}
Let $N$ be an integer and $(\lambda_{1},\ldots,\lambda_{N})$ a sequence of $N$ real numbers. For every $\varepsilon>0$, there 
exists an increasing sequence $(v_n)_{n \in \mathbf{N}}$, such that $\lim_{n\to \infty}v_{n}=+\infty$ and 
$|e^{\mathrm{i}\lambda_j v_n}-1|<\varepsilon$,   
 for every $n$ in $\mathbf{N}
$, for every $j\leq N$.
\end{lemma}
%\marco{
%Actually we need a stronger version of the Poincar\'e recurrence Theorem. Indeed we use this Lemma in the proof of Lemma~\ref{LEM_tripoint_dim_finie}  to prove that
%there exists a sequence of positive real numbers $(\tau_n)_{n\in \mathbf{N}}$ such that
%$e^{\tau_n A^{(N)}}e^{v^1_n(\|u^\ast\|_{L^1})A^{(N)}}$ tends to 
%$e^{v^{\ast}(\|u^\ast\|_{L^1})A^{(N)}}$ as $n$ tends to 
%infinity.
%
%\begin{lemma}
%Let $(X,\Sigma,\mu)$ be a finite measure space and let $f\colon X\to X$ be a measure-preserving transformation. For any $E\in \Sigma$, the set of those points $x$ of $E$ such that $f^n(x)\notin E$ for all $n>0$ has zero measure.
%\end{lemma}
%
%I propose: \\
%1) we state PRT in the most general form  without proof. \\
%2) We do not state the PRT (and of course we do not provide the proof).\\
%Messieurs faites votre jeu!
%}
%
\begin{proof}
Consider the distance on the $N$-dimensional torus $\mathbf{T}^N$ defined by
$$
\begin{array}{lccl}
d:\!\!\!\!\!\!\!\!\!\!\!\!\!\!\!\!\!\!\!&\mathbf{T}^N \times \mathbf{T}^N & \!\!\!\to \!\!\!& \mathbf{R}\\
& \big ( (e^{\mathrm{i}g_j})_{1\leq j \leq N}, 
(e^{\mathrm{i}h_j})_{1\leq j \leq N}\big )& \!\!\!\mapsto\!\!\! & \sup_{1\leq j \leq N} |
e^{\mathrm{i} g_j}-e^{\mathrm{i} h_j}|.
\end{array}
$$
The torus $\mathbf{T}^N$ endowed with the distance $d$ is compact. Hence the sequence 
$(U_n)_{n\in \mathbf{N}}:=\big ((e^{\mathrm{i}\lambda_j n})_{j\leq N}\big )_{n\in \mathbf{N}}$ 
accumulates (at least) in one point that we denote
 $U_{\infty}:=(e^{\mathrm{i}\theta_j})_{1\leq j\leq N}\in \mathbf{T}^N$. 
 We construct a sequence $(w_n)_{n\in \mathbf{N}}$ of integers by 
induction, let $w_1$ be the smallest positive integer $n$ such that $d(U_n,U_\infty)<
\varepsilon/2$. Assuming  
$w_n$  known, we chose $w_{n+1}$  as the smallest positive integer $n$ larger that $w_n+
(w_n-w_{n-1})$ such that $d(U_{w_{n+1}},U_\infty)< \varepsilon/2$.

Finally, we define $v_n=w_{n+1}-w_n$. By construction, for every $n$, $v_n\geq n$ and
\begin{eqnarray*} 
 |e^{\mathrm{i}\lambda_j v_n}-1| & \leq & |e^{\mathrm{i}\lambda_j (w_{n+1}-w_n)}-1|\\
 & \leq & |e^{\mathrm{i}\lambda_j w_{n+1}}-e^{\mathrm{i}\lambda_j w_{n}}|\\
 & \leq & d(U_{n+1},U_n)\\
 & \leq & d(U_{n+1},U_\infty) + d(U_n,U_\infty)\leq  \varepsilon 
 \end{eqnarray*} 
 for every $1\leq j\leq N$.
\end{proof}

\begin{lemma}\label{LEM_tripoint_dim_finie}
 For every  $a, b \in \mathbf{R}$, $a<0<b$, for every $T>0$, for every integrable
{ function} $u^{\ast}:\mathbf{R}\to \mathbf{R}$, there exists a
sequence $(u_n)_{n\in \mathbf{N}}$ of piecewise constant functions
 $u_{n}: [0,T_n] \to \{a,0,b\}$
 such that $X^{u_n}_{(N)}(T_n,0)$ tends to  $X^{u^{\ast}}_{(N)}(T,0)$ as $n$
tends to infinity and $\|u_n\|_{L^1}=\|u^{\ast}\|_{L^1}$. If, 
moreover,
$u^{\ast}$ is non-negative, the sequence $(u_n)_{n\in \mathbf{N}}$ can be chosen
such that $u_n$ takes value in $\{0,b\}$ for every $n$.
\end{lemma}

\begin{remark}
The approximation result in Lemma~\ref{LEM_tripoint_dim_finie} is classical and
 can be obtained, for instance, with Lie groups techniques, see \cite{such}. 
The novelty of Lemma~\ref{LEM_tripoint_dim_finie} is that 
the approaching sequence $(u_n)_n$ is bounded in $L^1(\mathbf{R},\mathbf{R})$. 
This point is crucial for the derivation of the infinite dimensional results in  
Section~\ref{SEC_infinite} below.
\end{remark}

\begin{proof}[Proof of Lemma~\ref{LEM_tripoint_dim_finie}]
To simplify the notation for every $u\in PC$, define the time-varying $N\times N$ matrix
$t\mapsto M_u(t)$ the entry $(j,k)$ of which is given 
by
$$
m_{jk}:t\mapsto\mathrm{sign}(u \circ v )(t) b_{jk} e^{ \mathrm{i}(\lambda_{j} - \lambda_{k}) v(t)},
$$
where 
$v$ is the cumulative function of $\mathcal{P}|u|$
vanishing at $0$, that is $v(t) =
\int_{0}^{t}\mathcal{P}|u|(s)\mathrm{d}s$. Notice that %$v$ may be defined only almost everywhere,  however 
$u\circ v$ is defined everywhere on 
$\left \lbrack 0,\|u\|_{L^1} \right \rbrack$. 

By density (for the $L^1$ norm) of the set $PC$ in 
$L^{1}(\mathbf{R})$ functions, 
one may assume without loss of generality that $u^\ast$ is piecewise constant not vanishing 
in $[0,T]$.  Let $v^{\ast}(t) = \int_{0}^{t}\mathcal{P}|u^{\ast}|(s)\mathrm{d}s$. By construction, $\int_0^{v^{\ast}(t)} |u^\ast(s)|\mathrm{d}s=t$ for every $t$ in $[0,\|u^\ast\|_{L^1}]$.
The solution $y^\ast$ of $\dot y=M_{u^{*}}y$ with initial condition $y(0) =I_{N}$
satisfies, by~\eqref{eq:0101}, the following relation
\begin{eqnarray}
e^{v^{\ast}(t)A^{(N)}}\!\!\!y^\ast(t)&=&\widehat{X}_{(N)}^{u^\ast} (t,0) \nonumber \\
&=&\!X^{u^\ast}_{(N)}(v^\ast(t),0) \label{EQ_reparam_decomp_phase_module}
\end{eqnarray}
for every $t$ in $\left \lbrack 0,\|u^\ast\|_{L^1} \right \rbrack$.

%Let $\lambda_j$ be defined by $A \phi_j= \mathrm{i} \lambda_j \phi_j$ for every $j$ in $\mathbf{N}$.

Consider, for every $\eta >0$ and $r\in \mathbf{R}$ the set
\begin{eqnarray*}
E_{\eta}(r)= \{v\in \mathbf{R}\mid |e^{\mathrm{i}\lambda_{j} r} - e^{\mathrm{i}\lambda_{j} v}| <\eta\\
 \mbox{ for every } 1 \leq j\leq N \}.
\end{eqnarray*}
 For every $r \in \mathbf{R}$, $E_{\eta}(r)$ is open and
 nonempty. Note that 
 $$
 |e^{\mathrm{i}\lambda_{j} r} - e^{\mathrm{i}\lambda_{j} v}|=2\left | \sin \left ( \frac{|\lambda_j||r-v|}{2}\right )\right |,
$$ 
thus each connected component of $E_{\eta}(r)$ has measure at least 
$$
\frac{\eta}{\sup_{1\leq j,k\leq N}|\lambda_j|}.
$$
Moreover, by Lemma~\ref{LEM_recurrence}, there exists an increasing  sequence $(v_n)_{n\in \mathbf{N}}$ of integers tending to $+\infty$, such that,
for $1\leq j,k \leq N$, $|e^{\mathrm{i}\lambda_{j}v_n} - 1| <\eta$ or, equivalently, 
$|e^{\mathrm{i}\lambda_{j} (r+v_n)} -e^{\mathrm{i}\lambda_{j} r} | <\eta$. 
 Hence, for every $n$ in $\mathbf{N}$, $r+v_n$ belongs to $E_{\eta}(r)$, which is not bounded from above. The same 
 argument shows that  $E_{\eta}(r)$ contains also $r-v_n$ and that it is not bounded from below.

For every $l>0$, let $v^{\ast}_l = \sum_{j=1}^{p_l} v_{l,j}
\chi_{[t_{l,j},t_{l,j+1})}$ be a piecewise constant  approximant of $v^{\ast}$ such that
$\|v^{\ast}_l  - v^{\ast}\|_{\infty} \leq l$ on $[0,\|u^{\ast}\|_{L^1}]$ and
such that the sign of $u^{\ast} \circ v^{\ast}_l$ is
constant on every interval $[t_{l,j},t_{l,j+1})$.  For every $\eta>0$, there
exists a (possibly discontinuous) piecewise affine function $v_{l}^{\eta}$
defined on every interval $[t_{l,j},t_{l,j+1})$ by
$$
 \dot v_{l}^{\eta} = 
\left\{
\begin{array}{ll}
1/b&  \mbox{ if }  u^{\ast} (v_{l,j})>0, \\
1/a &  \mbox{ if } u^{\ast} (v_{l,j})<0, \\
\end{array}
\right.
$$
and
$$
v_{l}^{\eta}(t) \in E_{\eta}(v_{l,j}) \quad  \mbox{ for } t \in [t_{l,j},t_{l,j+1}).
$$
Thus  $v_{l}^{\eta}$ is increasing (respectively decreasing) on $(t_{l,j},t_{l,j+1})$  if $ u^{\ast} (v_{l,j})>0$ (respectively $u^{\ast} (v_{l,j})<0$), see Figure \ref{fig:constructionv}.
\begin{figure}[tb!]
 \centering
 \includegraphics[width=8cm,height=4cm]{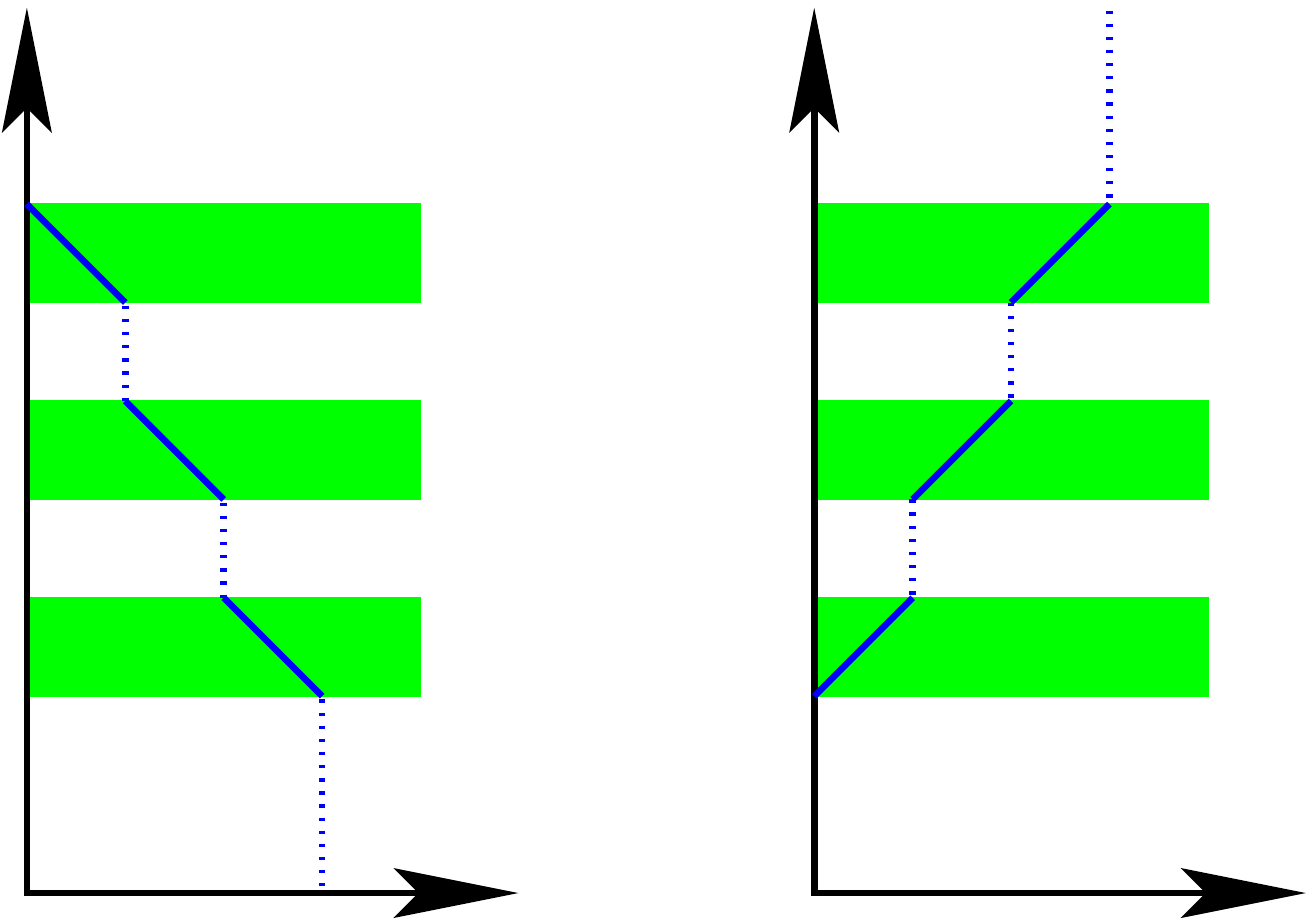}
 % Constructionv.: 357x294 pixel, 80dpi, 11.33x9.33 cm, bb=0 0 321 265
   \caption{Construction of the function $v_{l}^\eta$, when
$u^{\ast}(v_{l,j})<0$ (left) and $u^{\ast}(v_{l,j})>0$ (right).  The set
$E_{\eta}(v_{l,j})$ is coloured. The piecewise affine function $v_{l}^{\eta}$ is
discontinuous, with derivative equal to $1/a<0$ (left) or $1/b>0$ (right).
Notice that $v_{l}^{\eta}$ is injective in both cases.  The derivative 
$u_l^\eta$ of the  reciprocal function  of $v_{l}^{\eta}$ is piecewise affine
and takes value in $\{a,0,b\}$.}
 \label{fig:constructionv}
\end{figure}

By construction, the function $v_{l}^{\eta}$ is one-to-one on 
$(t_{l,j},t_{l,j+1})$. 
Its inverse on $(t_{l,j},t_{l,j+1})$, say $w_{l}^{\eta}$,  is a piecewise
affine function. 
%Notice that the latter could be not defined on some intervals 
%properly contained in 
%$(w_{l}^\eta(t_{l,j}),w_{l}^\eta(t_{l,j+1}))$ in the increasing 
%case or $(w_{l}^\eta(t_{l,j+1}),w_{l}^\eta(t_{l,j}))$ in the decreasing case. 
%Let $[a,b]$ any  interval (of maximal length) in which $w_{l}^{\eta}$ is not 
%defined and notice that, by construction, $\lim_{t\to a^{-}}w_{l}^{\eta}(t) = 
%\lim_{t\to b^{+}}w_{l}^{\eta}(t)$. We extend then $w_{l}^{\eta}$ by continuity 
%on the whole interval  $(t_{l,j},t_{l,j+1})$ by setting $w_{l}^{\eta}(t) =  
%\lim_{t\to a^{-}}w_{l}^{\eta}$ on every interval $[a,b]$ in which the inverse 
%of $v_{l}^{\eta}$ is not defined. 
The derivative $u_{l}^{\eta}$ of the continuous piecewise linear function  $w_{l}^{\eta}$  is a piecewise constant
function taking value in $\{a,0,b\}$. 

Moreover, by construction $\|u_{l}^{\eta}\|_{L^1} = 
\|u^{\ast}\|_{L^1}$.

For every $n$ in $\mathbf{N}$, let $ u_{n} = u_{l}^{\eta}$ with $l=\eta = 1/n$, let
$ v_n$ be the (possibly discontinuous) 
inverse function of $t\mapsto \int_0^t | u_n(s)|\mathrm{d}s$, and $ y_n$ the associated solution 
of 
$\dot y = M_{ u_{n}} y$ with initial condition $y(0)=I_N$.

For every $t$, $\int_0^t M_{ u_n}(\tau) \mathrm{d}\tau$ tends 
to $\int_0^t M_{u^{\ast}}(\tau) \mathrm{d}\tau$ 
as $n$ tends to infinity, uniformly on  $\left \lbrack 0,\|u^\ast\|_{L^1} \right \rbrack$. 
By \cite[Lemma 8.2]{book}, the associated solution $ y_{n}$ tends uniformly 
on $[0,\|u^\ast\|_{L^1}]$ to $y^{\ast}$. In particular, $ y_{n}(\|u^\ast\|_{L^1})$ converge toward 
$y^{\ast}(\|u^\ast\|_{L^1})$ as $n$ tends to infinity.

From (\ref{EQ_reparam_decomp_phase_module}), we have that for every $t$ in $[0,\|u^\ast\|_{L^1}]$,
\begin{eqnarray}
\lefteqn{\|X^{u^\ast}_{(N)}(v^\ast(t),0)-X^{u_n}_{(N)}(v_n(t),0)\|} \nonumber\\
&\leq &\|e^{v^{\ast}(t)A^{(N)}}y^\ast(t)-e^{v_n(t)A^{(N)}}y_n(t) \| \nonumber\\
&\leq &\|y^\ast(t)-y_n(t) \|+ \left \|e^{v^{\ast}(t)A^{(N)}}-e^{v_n(t)A^{(N)}} 
\right\|.
\label{EQ_major_diff_prop_dim_finie}
\end{eqnarray}
Taking $t=\|u^\ast\|_{L^1}$ in (\ref{EQ_major_diff_prop_dim_finie}) concludes 
the first part of the proof.

Finally, notice that if {$u^{\ast}\geq 0$}, then $u^{\ast}(v_{l,j})$
is always nonnegative, hence $v_{\eta}^{\eta}$ is increasing and $u_n$ takes only the values $0$ and $b$.
\end{proof}

\section{INFINITE DIMENSIONAL SYSTEMS}\label{SEC_infinite}

\subsection{Energy estimates for weakly-coupled quantum systems}\label{SEC_HIGH_NORM_WC}

If $(A,B,C,k)$ satisfies Assumption \ref{ASS_1}, $(A,B,C)$ is  \emph{$k$-weakly-coupled}.  We present here some properties of these
systems and refer to \cite{weakly-coupled} for further details.

The notion of weakly-coupled systems is closely related to the growth of the
 $k/2$-norm $\|\psi\|_{k/2}=\langle |A|^k \psi, \psi \rangle$. For $k=1$, this quantity is the expected 
 value of the energy of the system.
 Next result is a direct application of~\cite[Proposition~2]{weakly-coupled}
\begin{proposition}\label{PRO_croissance_norme_A} Let  $(A,B,C,k)$ satisfy Assumption 
\ref{ASS_1}.  Then,
for every $\psi_{0} \in D(|A|^{k/2})$, $K>0$,
$T\geq 0$, and $u$ piecewise constant such that
$\|u\|_{L^1}+\|u\|^2_{L^2}< K$, one has
\begin{equation}\label{eq:0505}
\left\|\Upsilon^{u}_{T}(\psi_{0})\right\|_{k/2} \leq e^{c(A,B,C,k) K} \| \psi_0 \|_{k/2}.
\end{equation}
\end{proposition}
Equation~\eqref{eq:0505} allows to define the solutions of~\eqref{EQ_main} for controls $u$ that are not necessarily piecewise constant. Indeed, let 
$u$ be in $L^1(\mathbf{R},\mathbf{R})\cap L^2(\mathbf{R},\mathbf{R})$ with support in 
$[0,T]$ for some $T>0$. There exists a sequence $(u_n)_{n\in \mathbf{N}}$ of piecewise 
constant functions with support in $[0,T]$ such that $\|u_n\|_{L^1}\leq \|u\|_{L^1}$ and 
$\|u_n\|_{L^2}\leq \|u\|_{L^2}$ for every $n$ in $\mathbf{N}$ and  the sequence $(u_n)_{n 
\in \mathbf{N}}$ tends to $u$ both in $L^1$ and in $L^2$ norm. Next result then guarantees convergence of the propagators.
\begin{lemma}\label{LEM_Propagateur_Cauchy}
Let $(u_{n})_{n\in \mathbf{N}}$ be a Cauchy sequence of piecewise constant functions 
both in $L^{1}$ and 
$L^{2}$, then
for every $t$ in $\mathbf{R}$ and every $\psi$ in $D(A)$, the sequence 
$(\Upsilon^{u_n}_{t,0}\psi)_{n \in \mathbf{N}}$ is a Cauchy sequence.
\end{lemma}

\begin{proof}
For the sake of simplicity, we define $x_n:t\mapsto \Upsilon^{u_n}_{t,0}\psi$. Since $\psi$ 
belongs to  the common domain $D(A)$ of the operators $D(A+\alpha B+\alpha^2C),$ for $\alpha \in 
\mathbf{R}$, the continuous mapping $x_n$ is a strong solution of (\ref{EQ_main}), see 
\cite{Marius}. Hence, $x_n$ is differentiable almost everywhere, $x_n(t)=x_n(0)+\int_0^t\dot x_n(s)\mathrm{d}s$ for every $t$ in $\mathbf{R}$ where $\dot{x}_n(t)=Ax_n(t) +u_n B x_n(t) + u_n^2 C x_n(t)$ for almost every $t$ 
in $\mathbf{R}$.

Let $n, m$ in $\mathbf{N}$. The continuous mapping $x_n-x_m$ is differentiable almost everywhere and, for almost every 
$t$ in $\mathbf{R}$, 
\begin{align*}
\frac{{d}}{{d}t}(x_n-x_m)\big |_{t}&\!=\! A(x_n-x_m)(t)+(u_n(t)-u_m(t))Bx_n(t)\\
& \quad  + u_m(t) B(x_n(t)-x_m(t)) \\
& \quad +(u_n^2(t)-u_m^2(t))C x_n(t) \\
& \quad + u_m^2(t) C(x_n(t)-x_m(t)) 
\end{align*}
By Duhamel formula, for every $t$ in $\mathbf{R}$, 
\begin{align}
\|(x_n-x_m)(t)\|= & \big \| \int_0^t \Upsilon^{u_m}_{t,s} (u_n(s)-u_m(s))Bx_n(s)\nonumber\\
&\quad  +(u_n^2-u_m^2)(s)C x_n(s)) \mathrm{d}s\big \| \nonumber\\
\leq & \|u_n-u_m\|_{L^1} \sup_{s \in \mathbf{R}}\|B x_n(s)\|\nonumber\\
&\quad + \|u_n^2-u_m^2\|_{L^1} \sup_{s \in \mathbf{R}}\|C x_n(s)\|\label{eq:0314}
\end{align}
By Proposition \ref{PRO_croissance_norme_A}, if $\|u\|_{L^{1}} + \|u\|^{2}_{L^{2}} < K$ then
$$
\sup_{s \in \mathbf{R}}\||A|^\frac{k}{2} x_n(s)\|\leq e^{c(A,B,C,k)K}\||A|^\frac{k}{2}
\psi\|.
$$ 
Notice, and this is crucial for the result, that the RHS does not depend on $n$.
By  Assumption \ref{ASS_1}.\ref{ASS_BC_Ak_borne}, 
$  \sup_{n\in \mathbf{N}}\sup_{s \in \mathbf{R}}\|B x_n(s)\|<+\infty$ and $ \sup_{n\in \mathbf{N}} \sup_{s \in 
\mathbf{R}}\|C x_n(s)\|<+\infty$. 

Since $(u_n)_{n\in \mathbf{N}}$ is a Cauchy sequence for the norms $L^1$ 
and $L^2$ then $\lim_{N \to \infty} \sup_{n,m\geq N} \|u_n-u_m\|_{L^1}=0$ and 
\begin{eqnarray*}
\lefteqn{\lim_{N \to \infty} \sup_{n,m\geq N} \|u_n^2-u_m^2\|_{L^1}}\\
&\leq&
 \lim_{N \to \infty} \sup_{n,m\geq N} \|u_n-u_m\|_{L^2}\|u_n+u_m\|_{L^2}\\
 &\leq& 
2 \lim_{N \to \infty} \sup_{n,m\geq N}  \|u\|_{L^2}\|u_n-u_m\|_{L^2}=0,
 \end{eqnarray*}
 hence, by~\eqref{eq:0314} we have $ \lim_{N \to \infty} \sup_{n,m\geq N}\|x_n(t)-x_m(t)\|=0$.
\end{proof}

Thanks to Lemma \ref{LEM_Propagateur_Cauchy} and to the completeness of the Hilbert space $H$, one can define $
\Upsilon^u_{t,0}\psi$ for $\psi$ in $D(A)$ as the limit of $\Upsilon^{u_n}_{t,0}\psi$ as $n$ tends to infinity. 
Notice that this limit is 
independent on the chosen approaching sequence $(u_n)_{n \in \mathbf{N}}$ . 
For every $t\geq 0$, the mapping $\psi \mapsto \Upsilon^u_{t,0}\psi$ admits a unique unitary extension on 
$H$. We can therefore define the propagator associated with a control $u$ which is both $L^{1}$ and $L^{2}$, 
as summed up in the following result.
\begin{proposition}\label{PRO_well_posedness}
Let $(A,B,C,k)$ satisfy Assumption \ref{ASS_1}. The mapping $u\mapsto \Upsilon_{\cdot ,0}^{u,(A,B,C)}$ which 
associates with every piecewise constant function a continuous curve of unitary transformations of $H$ 
bounded for the 
$\|\cdot \|_k$ norm admits a unique continuous extension for the $\|\cdot\|_{L^1} +\|\cdot\|_{L^2}$-norm. 
\end{proposition}
Thanks to Proposition \ref{PRO_well_posedness}, one can extend the result of
 Proposition~\ref{PRO_croissance_norme_A} 
to functions in $L^1(\mathbf{R}) \cap L^2(\mathbf{R})$. Another application (instrumental in our  study)
  of Proposition~\ref{PRO_croissance_norme_A} 
is the following approximation result, based on~\cite[Theorem~4]{weakly-coupled}.

\begin{proposition}\label{prop:gga}
 Let $k$ in $\mathbf{N}$ and $(A,B,C,k)$  satisfy Assumption \ref{ASS_1}.
Then
for every $\varepsilon > 0 $, $s<k$, $K\geq 0$, $n\in \mathbf{N}$, and
$(\psi_j)_{1\leq j \leq n}$ in $D(|A|^{k/2})^n$
there exists $N \in \mathbf{N}$
such that
for every piecewise constant function $u$ we have that
$$%\begin{equation}
\|u\|_{L^{1}} +\|u\|^{2}_{L^2} <  K \Rightarrow\| \Upsilon^{u}_{t}(\psi_{j}) -
X^{u}_{(N)}(t,0)\pi_{N} \psi_{j}\|_{s/2}  <  \varepsilon,
$$%\end{equation}
for every $t \geq 0$ and $j=1,\ldots,n$.
\end{proposition}
\begin{proof}
The result for $u$ piecewise constant is given by~\cite[Theorem~4]{weakly-coupled}. Then, by density, 
(see Proposition \ref{PRO_well_posedness}), the result holds true for general 
$u$ in $L^1(\mathbf{R},\mathbf{R}) \cap L^2(\mathbf{R},\mathbf{R})$. 
\end{proof}

\begin{remark}
In Propositions~\ref{PRO_croissance_norme_A} and~\ref{prop:gga}, the upper bound of the 
$|A|^{k/2}$ norm of the solution of (\ref{EQ_main}) or the bound on the error between 
the infinite dimensional 
system and its finite dimensional approximation only depend on the $L^1$ 
and $L^2$ norms of the control, not on the time. 
\end{remark}

\subsection{An infinite dimensional tracking result}\label{SEC_infinite_tracking}

Proposition~\ref{prop:gga} allows to adapt finite dimensional results to infinite dimensional systems. Here we present a sort of ``Bang-Bang'' Theorem for infinite dimensional systems.

\begin{lemma}\label{lem:tripoint}
Let $(A,B,0,k)$ satisfy Assumption \ref{ASS_1} with $k$ in $\mathbf{N}$, $T$ be a positive number, $a,b$
be two real numbers such that $a<0<b$,
$u^{\ast}$  be a locally integrable function with support in $[0,T]$,
and $N$ be an integer.
Then, 
for every $\varepsilon > 0 $, there exists a piecewise constant control $u_{\varepsilon} : [0,T_{\varepsilon}] \to \{a,0,b\}$  such that, for every $j\leq N$,
$
\|{\Upsilon^{u_{\varepsilon}}_{T_{\varepsilon},0}(\phi_j) - \Upsilon^{u^{\ast}}_{T,0}(\phi_j)}\| < \varepsilon$,
 and  $\|u_{\varepsilon}\|_{L^{1}} \leq \|u^{\ast}\|_{L^1}$.
Moreover, if $u^{\ast}$ is positive, then $u_{\varepsilon}$ may be chosen with value in $\{0,b\}$.
\end{lemma}

\begin{proof}
Let $\varepsilon>0$. By Proposition \ref{prop:gga}, there exists $N$ in $\mathbf{N}$ such that, for every piecewise constant function $u$ and for every $j\leq N$,
$$ 
\|u\|_{L^{1}} \leq \|u^{\ast}\|_{L^{1}} \Rightarrow\| \Upsilon^{u}_{t}(\phi_j) -
X^{u}_{(N)}(t,0)\pi_{N} \phi_j\| < \varepsilon. 
$$
From Lemma~\ref{LEM_tripoint_dim_finie}, there exists
$u_{\varepsilon}:[0,T\varepsilon]\rightarrow \{a,0,b\}$ piecewise constant such
that $\|u_\varepsilon\|_{L^1}\leq \|u^{\ast}\|_{L^1}$ and
$$
\|X^{u^\ast}_{(N)}(T,0) -X^{u_{\varepsilon}}_{(N)}(T,0)\|<\varepsilon.
$$
Then, for every $j\leq N$,
\begin{eqnarray*}
 \lefteqn{\|\Upsilon^{u_{\varepsilon}}_{T_{\varepsilon},0}(\phi_j) - \Upsilon^{u^{\ast}}_{T,0}(\phi_j)\|}\\
&\leq & \|\Upsilon^{u_{\varepsilon}}_{T_{\varepsilon},0}(\phi_j) - X^{u_\varepsilon}_{(N)}(t,0)\pi_{N} \phi_j \|\\
& & + \|X^{u_\varepsilon}_{(N)}(T_\varepsilon,0)\pi_{N} \phi_j -X^{u^\ast}_{(N)}(T,0)\pi_{N} \phi_j\|\\
&& + \|\Upsilon^{u^\ast}_{T,0}(\phi_j) - X^{u^\ast}_{(N)}(T,0)\pi_{N} \phi_j \|\\
&\leq & 3 \varepsilon. 
\end{eqnarray*}
The same proof shows that, if $u^{\ast}$ is positive, $u_\varepsilon$ can be chosen with values in $\{0,b\}$. 
\end{proof}
%\begin{remark}
%Notice 
%\end{remark}

\subsection{Simultaneous approximate controllability}\label{SEC_Infinite_proof}

We recall here the following result dealing with approximate controllability for bilinear systems, i.e. when $C=0$. Its proofs is given in~\cite[Theorem 2.11]{Schrod2}. 
\begin{thm}[\cite{Schrod2}]\label{THE_Control_collectively}
 Let $(A,B,0,0)$ satisfy Assumption \ref{ASS_1}.
 If there exists a non-resonant chain of connectedness of $(A,B,0)$
then, for every $N$ in $\mathbf{N}$, for every $\varepsilon>0$, for every $\delta>0$, 
for every unitary operator $\hat{\Upsilon}:H\to H$, 
there exists $T>0$ and a piecewise constant function $u:[0,T]\rightarrow [0,\delta]$ such  that
$\|\Upsilon^u_{T,0}\phi_j-\hat{\Upsilon}\phi_j\|<\varepsilon$, for every $j\leq N$.
\end{thm}

We now proceed  to the proof of the Theorem~\ref{THE_main}. 

\begin{proof}[Proof of Theorem~\ref{THE_main} (case $r=0$)]
Assume that $(A,B,C,k)$ satisfies Assumption \ref{ASS_1} for some $k$ in $\mathbf{N}$ and admits a strongly non-degenerate chain of connectedness. 
Then, there exists $\alpha>0$ such that $(A,B+\alpha C,0)$  satisfies Assumption \ref{ASS_1} 
and admits a  strongly non-degenerate chain of connectedness. 
By analyticity, this property is true for almost every $\alpha$ in $\mathbf{R}$.
From Theorem~\ref{THE_Control_collectively}, for every $N$ in $\mathbf{N}$, for every
unitary operator $\hat{\Upsilon}:H\to H$ 
for every $\varepsilon>0$, and for every $\delta>0$, 
there exist $T>0$ and a piecewise constant function $u:[0,T]\rightarrow [0,\delta]$ such  that
 $\|\Upsilon^{u,(A,B+\alpha C,0)}_{T,0}\phi_j-\hat{\Upsilon}\phi_j\|<\varepsilon$. 
By Lemma~\ref{lem:tripoint},
there exists $\tilde{u}:[0,T_{\tilde{u}}]\rightarrow \{0,\alpha\}$ such that
 $\|\Upsilon^{\tilde{u},(A,B+\alpha C,0)}_{T_{\tilde{u}},0}\phi_j- \Upsilon^{u,(A,B+\alpha C,0)}_{T,0}\phi_j\|<
 \varepsilon$.  
 Thus, for $j\leq N$, 
$\|\Upsilon^{\tilde{u},(A,B+\alpha C,0)}_{T_{\tilde{u}},0}\phi_j-\hat{\Upsilon}\phi_j\|<2\varepsilon$.
To conclude the proof of Theorem \ref{THE_main} for $r=0$, it is enough to notice that 
$\Upsilon^{\tilde{u},(A,B+\alpha C,0)}_{T_{\tilde{u}},0}=
 \Upsilon^{\tilde{u},(A,B,C)}_{T_{\tilde{u}},0}$, since  for every $t$, 
$\tilde{u}(t)B+\tilde{u}^2(t)C=\tilde{u}(t)(B+\alpha C)$ as $\tilde{u}$ takes  only the values $0$ and $\alpha$.
   \end{proof}

\subsection{Controllability between eigenstates}\label{SEC_RWA}
In this Section, we use averaging techniques to provide explicit expressions of control laws steering one eigenstate of the system to another in order to prove Theorems  \ref{THE_main_weak} and \ref{THE_main_12}.
 
Averaging methods consist in replacing an oscillating dynamics $\dot{y}=f(t) y$ by its average $\dot{z}= \bar{f} z$  where $\bar{f}=\lim \frac{1}{T}\int_0^T f(t)\mathrm{d}t$. When the dynamics $f$ is regular and small enough, the solutions $y$ and $z$ have similar behaviors.  Averaging theory has grown to a whole theory in itself. We refer to \cite{sanders} for an introduction.  In quantum mechanics, averaging theory has been extensively used (under the name of ``Rotating Wave Approximation'') since the 60's, for finite dimensional systems.  It has recently been extended to the case of infinite dimensional systems. In the following proposition, we restate~\cite[Theorem~1 and Section 2.4]{periodic} in our framework.

\begin{proposition}\label{PRO_RWA}
Let $(A,B,0,k)$ satisfy Assumption \ref{ASS_1}. Assume that $(p,q)$ is a weakly non-degenerate transition 
of $(A,B,0)$.
Define 
$ \mathcal{N}=\{n \in \mathbf{N}\,|\,$ there exists $(l_1,l_2) \mbox{ with } b_{l_1,l_2}\neq 0 \mbox{ and } 
|l_1-l_2|=n |\lambda_p-\lambda_q| \mbox{ and } \{l_1,l_2\}\cap \{p,q\}\neq \emptyset\}.
$
If $u$ and $u^{2}$ are locally integrable, $2\pi/|\lambda_p-\lambda_q|$-periodic and satisfies, 
for every $n$ in $\mathcal{N}$,
\begin{equation}\label{eq:conn1}
\int_0^{2\pi/|\lambda_p-\lambda_q|} 
\!\!\!\!\!\!\!\!e^{\mathrm{i}n|\lambda_p-\lambda_q| t} u(t) \mathrm{d}t  \neq 0
\quad \mbox{ if  } n=1
\end{equation}
and 
\begin{equation}\label{eq:connmag1}
 \int_0^{2\pi/|\lambda_p-\lambda_q|} \!\!\!\!\!\!\!\! e^{\mathrm{i}n|\lambda_p-\lambda_q| t} u(t) \mathrm{d}t  = 0
\quad \mbox{ if  } n>1
\end{equation}
then there exists $T^{\ast}>0$ such that $|\langle \phi_p, \Upsilon^{u^\ast/n,(A,B,0)}_{nT^{\ast},0}\phi_{q}\rangle |$ 
tends to $1$ as $n$ tends to infinity. Moreover,
$$
\lim_{n\to \infty}\frac{1}{n}\int_0^{nT^\ast}|u^\ast(t)|\mathrm{d}t\leq \frac{\pi}{2|b_{pq}|} \frac{\int_0^T |u^\ast(t)|\mathrm{d}t}{\left | \int_0^T u^\ast(t)\mathrm{d}t \right |}.
$$
\end{proposition}

Our aim is to extend the result of Proposition~\ref{PRO_RWA} to the case where $C\neq 0$.
\begin{proposition}\label{PRO_averaging}
 Let $(A,B,C,k)$ satisfy Assumption \ref{ASS_1}. Assume that $(p,q)$ is a weakly non-degenerate transition 
of $(A,B,0)$.
Define 
$ \mathcal{N}=\{n \in \mathbf{N}\,|\,$ there exists $(l_1,l_2) \mbox{ with } b_{l_1,l_2}\neq 0 \mbox{ and } 
|l_1-l_2|=n |\lambda_p-\lambda_q| \mbox{ and } \{l_1,l_2\}\cap \{p,q\}\neq \emptyset\}.
$
If $u$ and $u^2$ are locally integrable, $2\pi/|\lambda_p-\lambda_q|$-periodic and satisfy, 
for every $n$ in $\mathcal{N}$,
$$
\int_0^{2\pi/|\lambda_p-\lambda_q|} 
\!\!\!\!\!\!\!\!e^{\mathrm{i}n|\lambda_p-\lambda_q| t} u(t) \mathrm{d}t  \neq 0
\quad \mbox{ if  } n=1
$$
and 
$$
 \int_0^{2\pi/|\lambda_p-\lambda_q|} \!\!\!\!\!\!\!\! e^{\mathrm{i}n|\lambda_p-\lambda_q| t} u(t) \mathrm{d}t  = 0
\quad \mbox{ if  } n>1
$$
then there exists $T^{\ast}>0$ such that $|\langle \phi_p, \Upsilon^{u^\ast/n,(A,B,C)}_{nT^{\ast},0}\phi_{q}\rangle |$ 
tends to $1$ as $n$ tends to infinity.
\end{proposition}

\begin{proof}
For the sake of readability, we define $T:=\frac{2\pi}{|\lambda_p-\lambda_q|}$. 
Let $u$ be a locally integrable and square integrable $T$-periodic function satisfying~\eqref{eq:conn1} and~\eqref{eq:connmag1}.
By Proposition~\ref{PRO_RWA} there exists $T^{*}>0$ such that
$|\langle \phi_p, \Upsilon^{u^\ast/n,(A,B,0)}_{nT^{\ast},0}\phi_{q}\rangle | \to 1$ 
 as $n \to +\infty$.

Notice that, for every $n$ in $\mathbf{N}$, 
%\begin{eqnarray*}
%\int_0^{nT^\ast}\left|\frac{u(s)}{n}\right |\mathrm{d}s &\leq & \frac{1}{n} \left ( \frac{nT^\ast}{T} +1 \right ) \int_0^{T} \!\!\!|u(s)|\mathrm{d}s\\
%&\leq & \left ( \frac{T^\ast}{T} +1 \right ) \int_0^{T}\!\!\! |u(s)|
%\mathrm{d}s.
%\end{eqnarray*}
%Moreover, 
\begin{eqnarray}
\int_0^{nT^\ast}\left|\frac{u(s)}{n}\right |^2\mathrm{d}s &\leq & \frac{1}{n^2} \left ( \frac{nT^\ast}{T} +1 \right ) \int_0^{T}\!\!\! |u(s)|^2\mathrm{d}s \nonumber \\
&= & \left ( \frac{T^\ast}{nT} +\frac{1}{n^2} \right ) \int_0^{T} \!\!\!|u(s)|^2 \mathrm{d}s.\label{EQ_major_norme_L2_u}
\end{eqnarray}
By Proposition \ref{PRO_croissance_norme_A}, 
\begin{equation}
\sup_{n\in \mathbf{N}} \sup_{0\leq s, t\leq nT^\ast}\| \Upsilon^{u/n,(A,B,C)}_{s,t}\phi_q\|_{k/2}<+\infty,
\end{equation}
and, by Assumption \ref{ASS_1}.\ref{ASS_BC_Ak_borne}, 
\begin{equation}\label{EQ_major_norme_C}
 \sup_{n\in \mathbf{N}} \sup_{0\leq s,t\leq nT^\ast} \|C \Upsilon^{u/n,(A,B,C)}_{s,t}\phi_q\|<+\infty.
 \end{equation}

 Since $\phi_q$ belongs to $D(A)$, for every $n$ in $\mathbf{N}$ the mapping 
$t\mapsto \Upsilon^{u/n,(A,B,C)}_{t,0}\phi_q$ is a strong solution
of (\ref{EQ_main}). For every  $n \in \mathbf{N}$, by Duhamel formula we have,
\begin{eqnarray*}
\lefteqn{\left \|\Upsilon^{u/n,(A,B,C)}_{nT^\ast,0}\phi_q -\Upsilon^{u/n,(A,B,0)}_{nT^\ast,0}\phi_q \right \|}\\
&=& \left \|\frac{1}{n^2}\int_0^{nT^\ast}\!\!\!\!\!\!\!\!\!\! u^2(s)\Upsilon^{u/n,(A,B,0)}_{nT^\ast,s}C \Upsilon^{u/n,(A,B,C)}_{s,0}\phi_q 
\mathrm{d}s \right \|\\
& \leq &  \left ( \frac{1}{n^2}\int_0^{nT^\ast}\!\!\!\!\!\!\!\!\!\! u^2(s)\mathrm{d}s \right )   \sup_{n\in \mathbf{N}} \sup_{0\leq s,t\leq nT^\ast} \|C \Upsilon^{u/n,(A,B,C)}_{s,t}\phi_q\|
\end{eqnarray*}
From (\ref{EQ_major_norme_L2_u}) and (\ref{EQ_major_norme_C}), this last quantity tends to zero as $n$ tends to infinity, and Proposition \ref{PRO_averaging} follows from Proposition \ref{PRO_RWA}.  
\end{proof}

We now proceed to the proofs of Theorems \ref{THE_main_weak} and Theorems \ref{THE_main_12} in the case $r=0$.
\begin{proof}[Proof of Theorem \ref{THE_main_12} (case $r=0$)]
Let $\varepsilon>0$ and  $\delta>0$ such that $b_{pq}+\delta c_{pq}\neq 0$ be given 
and define $T=2\pi/|\lambda_p-\lambda_q|$.
 Using $u^\ast:t\mapsto 1+\sin ( t 2\pi/T)$ with the system $(A,B+\delta C,0)$,  Proposition \ref{PRO_RWA} states that 
 there exists $T^\ast$ such that 
 $|\langle \phi_p,  \Upsilon^{u^\ast/n, (A,B+\delta C,0)}_{nT^\ast,0} \phi_q \rangle| $
 tends to $1$ as $n$ tends to infinity. 
 %From \cite{periodic}, $\int_0^{nT^\ast}u^\ast(s)/n\mathrm{d}s\leq \pi/|b_{fg}+\delta c_{fg}|$.
 
By Assumption \ref{ASS_1}, the real number $\lambda_p$ is not zero. Hence
there exists a sequence $(t_{n})_{n\in \mathbf{N}}$ such that
$\|e^{t_n A}\Upsilon^{u^\ast/n, (A,B+\delta C,0)}_{nT^\ast,0} \phi_q-\phi_p\|$ tends to 
zero as $n$ tends to infinity. Notice that 
$$
e^{t_n A}\Upsilon^{u^\ast/n, (A,B+\delta C,0)}_{nT^\ast,0} \phi_q=\Upsilon^{w_n, (A,B+\delta C,0)}_{nT^\ast+t_n,0} \phi_q,
$$
where $w_n(s)=u^\ast(s)/n$ for $s\leq nT^\ast$ and $w_n(s)=0$ for $s\in (nT^\ast,nT^\ast+t_n)$.
 
 From Lemma~\ref{lem:tripoint}, for every $n$ in $\mathbf{N}$, there
 exists $u_n:[0,T_n]\to \{0,\delta \}$ such that $\| \Upsilon^{u_n,(A,B+\delta C,0)}_{T_n,0} \phi_q-
  \Upsilon^{w_n,(A,B+\delta C,0)}_{nT^\ast +t_n,0} \phi_q\|<\varepsilon$. Conclusion follows
  from the fact that
  $\Upsilon^{u_n,(A,B+\delta C,0)}_{T_n,0} \phi_q=\Upsilon^{u_n,(A,B,C)}_{T_n,0} \phi_q$,
  for every $n$ in $\mathbf{N}$.
\end{proof}

While primary oriented to the non-bilinear system (\ref{EQ_main}), Theorem \ref{THE_main_12} holds 
when $C=0$ and represents a slight improvement (by a factor $4/5$) of Proposition 2.8 in \cite{Schrod2}.

\begin{proof}[Proof of Theorem~\ref{THE_main_weak} (case $r=0$)]
Let $S$ be a weakly-non-degenerate chain of connectedness of $(A,B,C)$.
Theorem \ref{THE_main_weak} for $r=0$ is a consequence Theorem~\ref{THE_main_12} applied iteratively on every pair $(p,q)$ in $S$.
% by induction on the $S$-distance between $f$ and $g$ using Theorem \ref{THE_main_12}. 
% For every $f$ and $g$ in $\mathbf{N}$, the $S
%$-distance between $f$ and $g$ is the smallest $p$ such that there exits a sequence  
%$s_1=(s_1^1,s_1^2),s_2=(s_2^1,s_2^2),\ldots,s_r=(s_r^1,s_r^2) \in S$ for which $s_{1}^1 = f$, $s_{r}^2 = g$. 
\end{proof}
%%%%%%%%%%%%%%%%%%%%%%%%%%%%%%%%%%%%%%%%%%%%%%%%%%%%%%%%%%%%%%%%%%%%%%%%%%%%%%%%
\subsection{Approximate controllability in higher norms}\label{SEC_HIGH_NORM}
The proofs of Theorems~\ref{THE_main_weak} and~\ref{THE_main_12} for the general case $r>0$ are a consequence of an easy and well-known result of interpolation. We give a proof for the sake of completeness. 

\begin{lemma}\label{lem:interpolation}
Let $s<r$ be two real numbers, $(x_n)_{n \in \mathbf{N}}$ be a sequence that converges to zero in $H$ in $s$-norm and 
is bounded in $r$-norm. Then  $(x_n)_{n \in \mathbf{N}}$ tends to zero in $q$-norm for any $q<r$.
\end{lemma}
\begin{proof}
We first prove the result for $q<(r+s)/2$. For every $n$ in $\mathbf{N}$,
\begin{eqnarray*}
\|x_n\|_{\frac{s+r}{2}}^2&=&\langle |A|^\frac{s+r}{2} x_n, |A|^\frac{s+r}{2} x_n \rangle \\
&= &\langle |A|^{s} x_n,  |A|^{r} x_n \rangle \\
&\leq & \|x_n\|_s \sup_{n \in \mathbf{N}}\|x_n\|_{r},
\end{eqnarray*}
which tends to zero as $n$ tends to infinity. Replacing $s$ in the computation above by $(s+r)/2$ gives the result for $q<r-(r-s)/4$. After $N$ iterations of this process, the result is proved for any $q$ less than $r-(r-s)/2^N$ which  tends to $r$ as $N$ tends to infinity.
\end{proof} 
 
The general proof of the main results for the general case $r>0$ is then a consequence of this interpolation lemma, of Proposition~\ref{PRO_croissance_norme_A}, and of the uniform bound on the  $L^1$ and $L^2$ norm of the controls. 
Notice that the bound on the square of the $L^2$ norm of the control taking value in $\{0,\delta\}$ is exactly  $\delta$ times the  $L^1$ norm, since, for every $\delta$ in $\mathbf{R}$,  $u^2=\delta u$ if $u\in \{0,\delta\}$.
The three proof follows exactly the same strategy.

\begin{proof}[Proof of Theorem~\ref{THE_main}]
The sequence of propagators 
$
\Upsilon^{u_{\varepsilon}}_{T_{\varepsilon},0}\phi_j
$
tends to $\hat{\Upsilon}\phi_j$ in the norm of $H$. The sequence of controls 
$u_{\varepsilon}$ is bounded in the $L^{1}$ norm by~\cite[Remark~5.9]{Schrod2}, 
then we can apply Proposition~\ref{PRO_croissance_norme_A} to have a bound on 
the $k/2$-norm. The result then follows from Lemma~\ref{lem:interpolation}.
\end{proof}

\begin{proof}[Proof of Theorem~\ref{THE_main_12}]
The proof follows the proof of Theorem~\ref{THE_main} above. We prove that there exists a sequence of controls
$u_{\varepsilon}:
[0,T_{\varepsilon}]\rightarrow \{0,\delta\}$  such that
$\|u_{\varepsilon}\|_{L^1}\leq \pi/(|b_{pq}+ \delta c_{pq})$ and 
$\|\Upsilon^{u_{\varepsilon}}_{T_{\varepsilon},0}\phi_p-\phi_q\|$ tends to $0$ as $\varepsilon$ tends to $0$. Moreover the sequence $\Upsilon^{u_{\varepsilon}}_{T_{\varepsilon},0}\phi_p$ is bounded for the $k/2$-norm
by Proposition~\ref{PRO_croissance_norme_A} and Lemma~\ref{lem:interpolation} allows to conclude that
$\|\Upsilon^{u_{\varepsilon}}_{T_{\varepsilon},0}\phi_p-\phi_q\|_{r}$ tends to $0$ as $\varepsilon$ tends to $0$
for every $r<k/2$.
\end{proof}

\begin{proof}[Proof of Theorem~\ref{THE_main_weak}]
It is sufficient to notice that the bound on $L^{1}$-norm of the sequence of controls $u_{\varepsilon}$ is given by iteratively apply Theorem~\ref{THE_main_12} to every element of the connectedness chain connecting $p$ to $q$. The proof then follows from Proposition~\ref{PRO_croissance_norme_A} and Lemma~\ref{lem:interpolation} as in the proof of Theorems~\ref{THE_main} and~\ref{THE_main_12}.
\end{proof}

%%%%%%%%%%%%%%%%%%%%%%%%%%%%%%%%%%%%%%%%%%%%%%%%%%%%%%%%%%%%%%%%%%%%%%%%%%%%%%%%
\section{EXAMPLES}

\subsection{Bounded coupling  potentials}\label{SEC_bounded_potentials}

Let $\Omega$ be a compact Riemannian manifold or a bounded domain in $\mathbf{R}^n$. Let $V,W_1,W_2:\Omega \to \mathbf{R}$ be three measurable bounded 
functions. We consider the system
\begin{eqnarray}
\lefteqn{\mathrm{i}\frac{\partial \psi}{\partial t}(x,t)=
(-\Delta +V(x))\psi(x,t) + u(t) W_1(x)\psi(x,t)} \nonumber \\&& \quad \quad  \quad \quad 
\quad \quad  \quad \quad  \quad   + u^2(t) W_2(x) \psi(x,t),\quad \quad \label{EQ_bilinear_compact_borne}
\end{eqnarray}
with $x$ in $\Omega$ and $t$ in $\mathbf{R}$. This system has been studied in~\cite{Morancey} when $\Omega$ is a bounded 
domain of $\mathbf{R}^n$,  and the potentials $W_1$ and $W_2$ are $C^2$.

In order to apply our results, we define $H=L^2(\Omega,\mathbf{C})$, $A:\psi\in D(A) \mapsto 
\mathrm{i}(\Delta -V)\psi$,
$B:\psi \in L^2(\Omega,\mathbf{C}) \mapsto -\mathrm{i}W_1\psi$ and $C:\psi \in L^2(\Omega,
\mathbf{C}) \mapsto -\mathrm{i}W_2\psi$. 
By Kato-Rellich theorem, the domain $D(A)$ of $A$ is equal to 
$H^2_{(0)}=\{\psi \in H^2(\Omega,\mathbf{C})|\psi_{|\partial \Omega}=\Delta \psi_{|\partial 
\Omega}=0\}$, the domain of the Laplacian, if $\Omega$ is a bounded domain of $\mathbf{R}^n$ 
and equal to $H^2(\Omega,\mathbf{C})$ if $\Omega$ is compact manifold.
The operators $B$ and $C$ are bounded from $H$ to $H$ with norms 
$\|W_1\|_{L^\infty}$ and $\|W_2\|_{L^\infty}$, respectively.

We restrict ourselves to the generic case (see \cite{genericity-mario-paolo})
 where $A$ has only simple eigenvalues. 
Without further regularity assumptions on $W_1$ and $W_2$, it is not clear if $(A,B,C,k)$ 
satisfies Assumption~\ref{ASS_1} for any $k>0$. 

By standard regularization procedures for every $\eta>0$, there exist $W_{1,
\eta},W_{2,\eta}:\Omega \to \mathbf{R}$ such that 
(i)  $W_{1,\eta}$ $W_{2,\eta}$ are $C^2$ on $\Omega$, 
(ii) if $\Omega$ is a bounded domain of  $\mathbf{R}^n$, $W_{1,\eta}$ and $W_{2,\eta}$ tend to zero, with their two first  derivatives, on the boundary of $\Omega$, and 
(iii) $\|W_j-W_{j,\eta}\|_{L^1}\leq \eta$ for $j=1, 2$. 
The linear operators $B_\eta:\psi\mapsto W_{1,\eta} \psi$ and $C_\eta: \psi\mapsto W_{2,\eta} \psi$ are bounded from $D(A)$ to $D(A)$. By Proposition 8 of \cite{weakly-coupled}, $(A,B_\eta,C_\eta)$ is $1$-weakly-coupled or, equivalently,  
$(A,B_\eta,C_\eta,1)$ satisfies Assumption~\ref{ASS_1}. 
\begin{remark}
The definition of $\Upsilon^{u,(A,B_\eta,C_\eta)}$ depends on the choice of $B_\eta$ and $C_\eta$, which 
is not unique. 
\end{remark}
The key point of this section is the following observation.
\begin{lemma}\label{PRO_major_distance_regularization}
For every $\eta>0$, for every $u$ in $L^1(\mathbf{R},\mathbf{R})\cap L^2(\mathbf{R},\mathbf{R})$, for every $t$ in $
\mathbf{R}$, for every  $\psi$ in $H$, 
$$\|\Upsilon^{u,(A,B,C)}_{t,0} -\Upsilon^{u,(A,B_\eta,C_\eta)}_{t,0}\|\leq \eta (\|u\|_{L^1}+\|u\|
_{L^2}^2).$$
\end{lemma} 
Thanks to Lemma~\ref{PRO_major_distance_regularization}, 
we can apply the results above to system~\eqref{EQ_bilinear_compact_borne}.  For instance Theorem \ref{THE_main} applied to system~\eqref{EQ_bilinear_compact_borne} reads.
\begin{proposition}\label{PRO_THEO1_cas_borne}
Assume that $(A,B,C)$ admits a strongly non-degenerate chain of connectedness. Then, for every $
\varepsilon>0$, for every unitary $\hat{\Upsilon}:H\to H$, for every $l$ in $\mathbf{N}$, for almost every $\alpha>0$
there exists a piecewise constant function $u_\varepsilon:[0,T_\varepsilon]\to \{0,\alpha\}$ such that
$\|\Upsilon^{u_\varepsilon}_{T_\varepsilon,0} \phi_j -\hat{\Upsilon}\phi_j \|< \varepsilon$, for every $j \leq l$.  
\end{proposition}
\begin{proof}
For every $\alpha>0$ such that $S$ is a strongly non-degenerate chain of connectedness of $(A,B+\alpha 
C,0)$, by Theorem 
\ref{THE_Control_collectively}, there exists a piecewise constant function $u:[0,T]\to [0,\alpha]$ such that 
%the propagator $\Upsilon^{u,A,B+\alpha C}$ of $x'=Ax +u(t) (B+\alpha C) x$ satisfies 
$\|\Upsilon^{u, (A, B+\alpha C,0)}_{T,0} 
\phi_j -\hat{\Upsilon}\phi_j \|< \varepsilon/3$, for every $j\leq l$. Define 
$$
\eta=\frac{1}{3}\frac{\varepsilon}{\|u\|_{L^1}(1+\alpha)}.
$$
As before choose 
$W_{1,\eta},W_{2,\eta}:\Omega \to \mathbf{R}$ such that 
(i)  $W_{1,\eta}$ $W_{2,\eta}$ are $C^2$ on $\Omega$, 
(ii) if $\Omega$ is a bounded domain of  $\mathbf{R}^n$, $W_{1,\eta}$ and $W_{2,\eta}$ tend to zero, with their two first  derivatives, on the boundary of $\Omega$, and 
(iii) $\|W_j-W_{j,\eta}\|_{L^1}\leq \eta$ for $j=1, 2$. 
Then the  linear operators $B_\eta:\psi\mapsto W_{1,\eta} \psi$ and $C_\eta: \psi\mapsto W_{2,\eta} \psi$
satisfy
$\|B-B_\eta\|<\eta$, $\|C-C_\eta\|<\eta$
and $(A,B_\eta,C_\eta,1)$ satisfies Assumption~\ref{ASS_1}.

By Lemma \ref{lem:tripoint}, there exists a piecewise constant function 
$u_\varepsilon:[0,T_\varepsilon]\to \{0,\alpha\}$ such that $\|u_\varepsilon\|_{L^1}\leq \|u\|_{L^1}$ and
$\|\Upsilon^{u_\varepsilon, (A, B_\eta+\alpha C_\eta,0)}_{T_\varepsilon,0} 
\phi_j -\Upsilon^{u, (A, B_\eta+\alpha C_\eta,0)}_{T,0} 
\phi_j \|< \varepsilon/3$, for every $j\leq l$.

Notice that 
$$
\Upsilon^{u_\varepsilon, (A, B_\eta+\alpha C_\eta,0)}_{T_\varepsilon,0}= \Upsilon^{u_\varepsilon, (A, B_\eta, C_\eta)}_{T_\varepsilon,0},
$$
and   $\|u_\varepsilon\|_{L^2}^2=\alpha \|u_\varepsilon\|_{L^1}$ since $u_\varepsilon$ takes value in 
$\{0,\alpha\}$.

Finally, for every $j\leq l$,
\begin{eqnarray}
\lefteqn{\|\Upsilon^{u_\varepsilon,(A,B,C)}_{T_\varepsilon,0} \phi_j -\hat{\Upsilon}\phi_j \| } \nonumber \\
 &\leq  & \|\Upsilon^{u_\varepsilon,(A,B,C)}_{T_\varepsilon,0} \phi_j - \Upsilon^{u_\varepsilon,(A,B_\eta,C_\eta)}_{T_
 \varepsilon,0} \phi_j\| \nonumber  \\
 && \quad + 
 \|\Upsilon^{u_\varepsilon,(A,B_\eta,C_\eta)}_{T_\varepsilon,0} \phi_j - \Upsilon^{u_\varepsilon,(A,B_\eta +\alpha C_
 \eta,0)}_{T_\varepsilon,0} \phi_j\| \nonumber  \\
 && \quad + 
 \|\Upsilon^{u_\varepsilon,(A,B_\eta +\alpha C_\eta,0)}_{T_\varepsilon,0} \phi_j
 -  \Upsilon^{u,(A,B_\eta +\alpha C_\eta,0)}_{T,0} \phi_j\| \nonumber \\
 &&\quad  + \| \Upsilon^{u,(A,B_\eta +\alpha C_\eta,0)}_{T,0} \phi_j -\hat{\Upsilon}\phi_j \| \\
 & \leq & \frac{\varepsilon}{3} + 0 + \frac{\varepsilon}{3} + \frac{\varepsilon}{3}=\varepsilon. 
\end{eqnarray}
Proposition \ref{PRO_THEO1_cas_borne} follows by observing that $S$ is a strongly non-degenerate chain of 
connectedness of $(A,B+\alpha C,0)$ for almost every $\alpha$ in $\mathbf{R}$, see Lemma 
\ref{LEM_perturb_chaine_connexite}.
\end{proof}

\subsection{Perturbation of the harmonic oscillator}\label{SEC_pert_harm_oscil}
The quantum harmonic oscillator is among the most important examples of quantum
system (see, for instance, \cite[Complement $G_V$]{cohen77}). Its controlled
bilinear version has been extensively studied (see, for instance, \cite{Rouchon,illner} and references therein).

We consider here a 1D-model involving, in addition to the standard bilinear term modeling a constant electric field, a 
Gaussian perturbation. Precisely, for given constant $a>0, b$, and $c$, 
the dynamics is given, for $x$ in $\mathbf{R}$, by: 
\begin{equation}\label{EQ_dyn_harm_osc_perturb}
\mathrm{i}\frac{\partial \psi}{\partial t}=(-\Delta +x^2)\psi + u(t) x \psi + u^2(t) e^{-
ax^2+bx+c} \psi
\end{equation}
With the notations of Section~\ref{SEC_framework} we have $H=L^2(\mathbf{R},\mathbf{C})$, 
$A:\psi \mapsto \mathrm{i}(\Delta-x^2)\psi$, 
$B:\psi\mapsto -\mathrm{i}x\psi$ and $C:\psi \mapsto -\mathrm{i}e^{-ax^2+bx+c}\psi$

A Hilbert basis of $H$ made of eigenvectors of $A$ is given by the sequence of
the Hermite functions $(\phi_n)_{n \in \mathbf{N}}$, associated with the
sequence $( - \mathrm{i} \lambda_n)_{n \in \mathbf{N}}$ of eigenvalues where
$\lambda_n=n-1/2$ for every $n$ in $\mathbf{N}$. In the basis $(\phi_n)_{n \in
\mathbf{N}}$, $B$ admits a tri-diagonal structure
$$
\langle \phi_j,B\phi_k\rangle = \left \{\begin{array}{cl}
- \mathrm{i} \sqrt{\frac k2} & \mbox{if } j=k-1,\\
- \mathrm{i}\sqrt{\frac{k+1}2} & \mbox{if } j=k+1,\\
0 & \mbox{otherwise.}
\end{array} \right.
$$
The operator $C$ couples most of the energy levels of $A$, see \cite[Proposition 6.4]{Schrod}.

For every $k$ in $\mathbf{N}$, the system $(A,B,0,k)$ satisfies Assumption \ref{ASS_1}  
 (see Section IV.E in \cite{weakly-coupled}) and
\begin{eqnarray*}
c_k(A,B,0) &\leq &3^{k}-1.
\end{eqnarray*}
For every $k$ in $\mathbf{N}$, a direct computation shows that $C$ is bounded from $D(|A|^k)$ to $D(|A|^k)$.
Hence, by Proposition 6 of \cite{weakly-coupled}, $(A,0,C,k)$ satisfies Assumption \ref{ASS_1} for every $k$.
Finally, $(A,B,C,k)$ satisfies Assumption \ref{ASS_1} for every $k$.

The quantum harmonic oscillator $(A,B,0)$ is not controllable (in any reasonable
sense) as proved in~\cite{Rouchon}. We aim at proving the following.
\begin{proposition}\label{PRO_contr_osc_harm_perturb}Assume that $\sqrt{1-a}$ and $b$ are algebraically independent.
 Then, for every $\varepsilon>0$, for every $j$ in $\mathbf{N}$, there exist $T>0$ and  
 a piecewise constant function $u:[0,T]\to \mathbf{R}$ 
 such that $\|\Upsilon^u_{T,0}\phi_1-\phi_j\|<\varepsilon$.
\end{proposition}
The main tool in the proof of Proposition~\ref{PRO_contr_osc_harm_perturb} is the following analytic perturbation argument (see Chapter VII of \cite{Kato}).
\begin{proposition}[\cite{Kato}]\label{PRO_perturbation_spectrale_analytique}
For every $\alpha$ in $\mathbf{R}$ and $n$ in $\mathbf{N}$, there exist two analytic mappings 
$\boldsymbol{\lambda}^\alpha_n:\mathbf{R}\to\mathbf{R}$ and 
$\boldsymbol{\phi}^\alpha_n:\mathbf{R}\to L^2(\mathbf{R},\mathbf{C})$ such that 
(i) for every $t$ in $\mathbf{R}$, $A+t(B+\alpha C)\boldsymbol{\phi}^\alpha_n(t)=-\mathrm{i}
\boldsymbol{\lambda}^\alpha_n(t)\boldsymbol{\phi}^\alpha_n(t)$; (ii) 
$\left .\frac{d}{dt}\boldsymbol{\lambda}^\alpha_n(t) \right |_{0}=b_{nn}+\alpha c_{nn}$; (iii) for every $t$ in $
\mathbf{R}$, $(\boldsymbol{\phi}^\alpha_n(t))_{n \in \mathbf{N}}$ is a Hilbert basis of $L^2(\mathbf{R},\mathbf{C})$;
(iv) $(\boldsymbol{\phi}^\alpha_n(0))_{n \in \mathbf{N}}=({\phi}_n)_{n \in \mathbf{N}}$.   
\end{proposition}

\begin{proof}[Proof of Proposition \ref{PRO_contr_osc_harm_perturb}]
From Proposition 6.4 of \cite{Schrod}, for every $n$ in $\mathbf{N}$, 
the pair $(n,n+1)$ is a strongly non-degenerate transition of 
$(A+\mu(B+2 \alpha C),B+\alpha C,0)$ for almost every $(\alpha,\mu)$ in $\mathbf{R}^2$.    

We proceed by induction. For $p=2$,
choose $\alpha$ and $\mu$ positive small enough such that, with the notations of Proposition  
\ref{PRO_perturbation_spectrale_analytique}, 
$\|\boldsymbol{\phi}^\alpha_j(\mu)-\alpha_j\|<\varepsilon/4$ for $j=1,2$, 
$|b_{12}+\alpha c_{12}|=\left | 1+ 
\frac{\alpha b e^{c-\frac{b^2}{4(a-1)}}}{\sqrt{2}(1-a)^{3/2}}\right |\neq 0$ and 
$\mu^2 +2 \mu + \mu \alpha < \frac{\varepsilon}{4\pi \| C\|}$
By Theorem \ref{THE_main_12}, there exists a piecewise constant function $v:[0,T]\to [0,1]$ such that 
$\|\Upsilon^{v,(A+ \mu(B+\alpha C), B, C)}_T\boldsymbol{\phi}^\alpha_1(\mu)-\boldsymbol{\phi}^\alpha_2(\mu)\|<
\varepsilon/4$. Then, defining $u:t\in [0,T] \mapsto v(t)+ \mu$:
\begin{eqnarray*}
\lefteqn{\|\Upsilon^{u,(A,B,C)}_{T,0}\phi_1-\phi_2\|}\\
&\leq & \|\Upsilon^{u,(A,B,C)}_{T,0}\phi_1-\Upsilon^{u,(A,B,C)}_{T,0}\boldsymbol{\phi}^\alpha_1(\mu)\| \\
&& \quad + \|\Upsilon^{u,(A,B,C)}_{T,0}\boldsymbol{\phi}^\alpha_1(\mu) -
\Upsilon^{v,(A+\mu(B+\alpha C),B,C)}_{T,0}\boldsymbol{\phi}^\alpha_1(\mu) \|\\
&& \quad +
\|\Upsilon^{v,(A+\mu(B+\alpha C),B,C)}_{T,0}\boldsymbol{\phi}^\alpha_1(\mu)-\boldsymbol{\phi}^\alpha_2(\mu)\| \\
&& \quad +
\|\boldsymbol{\phi}^\alpha_2(\mu) -\phi_2\|\\
&\leq & \frac{ \varepsilon}{4} + \frac{ \varepsilon}{4} + \frac{ \varepsilon}{4} + \frac{ \varepsilon}{4}.
\end{eqnarray*}
The general step is similar, replacing $b_{12}=-\mathrm{i}$ with $b_{n,n+1}=-\mathrm{i}\sqrt{(n+1)/2}$, and choosing $\alpha$ small enough such that $b_{n,n+1}+\alpha c_{n,n+1}\neq 0$.
\end{proof}

\section{CONCLUSIONS AND FUTURE WORKS}

\subsection{Conclusions}
In this analysis, we present a general approximate controllability result for 
infinite dimensional 
quantum systems when a
polarizability term is considered in addition to the standard dipolar one. For the important case of 
transfer between two eigenstates of the free Hamiltonian, simple periodic control laws may be used. 

\subsection{Future Works}
Many questions concerning the controllability of infinite dimensional quantum systems are still open. 
Among many other  topics, one can cite the extension of the controllability results to systems involving
better approximation of the external field, involving higher powers of the control, 
or the existence (and the estimation) of a minimal time needed to steer a quantum system from a 
given source to a given neighborhood of a given target.²

%%%%%%%%%%%%%%%%%%%%%%%%%%%%%%%%%%%%%%%%%%%%%%%%%%%%%%%%%%%%%%%%%%%%%%%%%%%%%%%%
\section{ACKNOWLEDGMENTS}
It is a pleasure for the third author to thank Karine Beauchard for interesting discussions about the 
methodology used in \cite{Morancey}.

%%%%%%%%%%%%%%%%%%%%%%%%%%%%%%%%%%%%%%%%%%%%%%%%%%%%%%%%%%%%%%%%%%%%%%%%%%%%%%%%
\bibliographystyle{IEEEtran}
\bibliography{biblio}

\end{document}